\documentclass[11pt,a4paper]{article}
\usepackage{amsmath}
\usepackage{amsthm}
\usepackage{makeidx}
\usepackage{latexsym}
\usepackage{graphicx}
\usepackage{amssymb}
\usepackage{hyperref}
\usepackage{float}
\newtheorem{theorem}{Theorem}
\newtheorem{example}{Example}

\newtheorem{proposition}{Proposition}
\newtheorem{corollary}{Corollary}
\newtheorem{remark}{Remark}
\newtheorem{definition}{Definition}
\theoremstyle{remark}

\usepackage[latin1]{inputenc}
\title{Colored compositions, Invert operator and elegant compositions with the "black tie"}
\author{Marco Abrate, Stefano Barbero, Umberto Cerruti, Nadir Murru}
\date{}

\begin{document}
\maketitle

\begin{abstract}
This paper shows how the study of colored compositions of integers reveals some unexpected and original connection with the Invert operator. The Invert operator becomes an important tool to solve the problem of directly counting the number of colored compositions for any coloration. The interesting consequences arising from this relationship also give an immediate and simple criterion to determine whether a sequence of integers counts the number of some colored compositions. Applications to Catalan and Fibonacci numbers naturally emerge, allowing to clearly answer to some open questions. Moreover, the  definition of colored compositions with the "black tie" provides straightforward combinatorial proofs to a new identity involving multinomial coefficients and to a new closed formula for the Invert operator. Finally, colored compositions with the ``black tie" give rise to a new combinatorial interpretation for the convolution operator, and to a new and easy method to count the number of parts of colored compositions.
\end{abstract}

\section{Colored compositions and the Invert operator}

The compositions of integer numbers correspond to ordered partitions \cite{Mac} in the following sense: any composition of an integer $n$ is a sequence of integers (called \emph{parts}) whose sum is $n$, univocally determined by the order of its parts. In \cite{Agra1} and \cite{Agra2}, $n$--colour partitions and $n$--colour compositions have been introduced, respectively. In a $n$--colour composition a part of size $m$ occurs with $m$ different colors, i.e., we have $m$ different parts of size $m$. Recently, colored compositions have been studied in different works, see, e.g., \cite{Agra3}, \cite{Drake}, \cite{Agra4}, \cite{Guo}.\\
\indent In the following, we extend the study to general \emph{colored compositions} where each part can occur with any number of colors (i.e., a part of size $m$ can occur with $j$ different colors, where $j$ can be any integer number). We give combinatorial interpretations of colored compositions from a different point of view by using the Invert operator, the ordinary complete Bell polynomials and linear recurrence sequences.\\
\indent Let us fix some definitions.
\begin{definition}
We define the \emph{coloration} $X$ to be the sequence $X=(x_i)_{i=1}^{\infty}$ of non--negative integers, where any $x_i$ is the number of colors of the part $i$. If $x_i=0$, it means that we do not use the integer $i$ in the compositions. Moreover, we define
\begin{itemize}
\item $L(X)$ as the set of all colored compositions with coloration $X$; we call an element of this set a \emph{composition} (for shortness) and indicate it with bold letters, e.g., $\textbf{b}\in L(X)$
\item $L_n(X)$ as the set of colored compositions of $n$ with coloration $X$
\item $A_n(X) = \lvert L_n(X) \rvert$; $A=A(X)=(A_n(X))_{n=1}^\infty$ is the sequence of the number of colored compositions of $n$ with coloration $X$, for $n=1,2,\ldots$
\item $p(\textbf{b})$ as the number of the parts of the composition \textbf{b}
\item $P_n(X) = \sum_{\textbf{b}\in L_n(X)} p(\textbf{b})$
\item $r(\textbf{b}) = p(\textbf{b})+1$ as the number of break--points
\item $R_n(X) = \sum_{\textbf{b}\in L_n(X)} r(\textbf{b})$  
\end{itemize}
\end{definition}
\begin{remark}
Let us observe that $A_n(X)$ can be viewed as a polynomial in $x_1,\ldots ,x_n$ (since parts of size greater than $n$ can not be used in the composition of $n$). Thus, sometimes we will write $A_n(x_1, \ldots, x_n)$ instead of $A_n(X)$. Moreover, we set $A_0(X)=1$, for any coloration $X$, meaning that we can compose the number 0 only if we do not use any composition. 
\end{remark}
\begin{definition}
The Invert operator $I$ transforms a sequence $a=(a_n)_{n=0}^\infty$ into a sequence $b=(b_n)_{n=0}^\infty$ as follows:
\begin{equation*} 
I(a)=b, \ \ \ \sum_{n=0}^{\infty}b_nt^n=\cfrac{\sum_{n=0}^{\infty}a_nt^n}{1-t\sum_{n=0}^{\infty}a_nt^n}. 
\end{equation*}
Sometimes, we will use the notation $(I(a))_n$ for the $n$--th term of the transformed sequence $b$.
\end{definition}
\begin{definition}\label{invert}
The \emph{Invert interpolated operator} $I^{(x)}$, with parameter $x \in \mathbb R$, transforms any sequence $a$, having ordinary generating function $a(t)$, into a sequence $b=I^{(x)}(a)$ having ordinary generating function
\begin{equation*}
b(t)=\cfrac{a(t)}{1-xta(t)}\quad.  
\end{equation*}
\end{definition}
\begin{definition}
The complete ordinary Bell polynomials are defined by (see, e.g., \cite{Port})
$$ B_n(t)=\sum_{k=1}^nB_{n,k}(t), $$
where $t=(t_1,t_2,\ldots)$ and the partial ordinary Bell polynomials $B_{n,k}(t)$ satisfy
$$ \left(\sum_{n\geq1}t_nz^n\right)^k=\sum_{n \geq k}B_{n,k}(t)z^n, $$
and
$$ 
B_{n,k} \left( t \right) = \sum_{\substack {i_1  + 2i_2  +  \cdots  + ni_n = n \\ i_1  + i_2  +  \cdots  + i_n  = k}} \frac{{k!}}{{i_1 !i_2 ! \cdots i_n !}}t_1^{i_1 } t_2^{i_2 }  \cdots t_n^{i_n }\quad.
$$
\end{definition}
\noindent From the definition we can see that $B_n(t)$ corresponds to the sum of monomials $t_{h_1}\cdots t_{h_k}$ (where every $t_{h_j}$ could be repeated) whose coefficients are the number of way to write the word $t_{h_1}\cdots t_{h_k}$, when $\sum_{j=1}^{k}h_j=n$. For example,
$$ B_4(t)=t_4+2t_1t_3+t_2t_2+3t_1t_1t_2+t_1t_1t_1t_1=t_4+2t_1t_3+t_2^2+3t_1^2t_2+t_1^4. $$
This combinatorial interpretation of the Bell polynomials  could be easily described in terms of colored compositions. It follows that
$$A_n(x_1,\ldots,x_n)=B_n(x_1,\ldots,x_n)$$
for $i=1,\ldots,n$ and $\forall n\geq1$. Moreover, in \cite{bcm} a connection between Invert operator and Bell polynomials has been highlighted. Given a sequence $a=(a_n)_{n=0}^\infty$, then $(I(a))_n=B_n(a_0,\ldots,a_{n-1})$. Consequently we have an important relation with colored compositions: thinking of a coloration $X$ as a sequence $a$ to which apply Invert operator and such that $a_n=x_{n+1}$ for all $n\geq 0$, the following equality holds
\begin{equation} \label{inva}
A(X)=I(X).
\end{equation}
This is a very interesting result that allows us to use the Invert operator in the study of colored compositions.
\begin{remark}
Let us observe that Eq. (\ref{inva}) holds for any coloration $X$, even in the case of occurrence of zeros in the sequence $X$. In fact we recall that, if the equality $x_i=0$  holds for some $i$, the length $i$ does not appear in the  compositions so the polynomial $A_n(X)$ is independent of the variable $x_i$.
\end{remark}
As an immediate consequence of  Eq. (\ref{inva}), for every coloration $X$ if we apply the Invert operator we obtain the sequence $A(X)$ corresponding to the number of colored compositions for any integer $n$.\\
\indent For example,  when $X=(1,2,3,4,\ldots)$, we can immediately find the known result $A_n(X)=F_{2n} \forall n \geq 1 $ (where $F_n$ are the Fibonacci numbers, see \cite{Agra2}), as in the classical $n$--colour compositions. Indeed, in this case the generating function of $X$ is $\cfrac{1}{(1-t)^2}$ (see, e.g., \cite{Plouffe}) and by definition of the Invert operator, the sequence $(A_n)_{n=1}^\infty=I(X)$ has generating function
$$\cfrac{\frac{1}{(1-t)^2}}{1-t\frac{1}{(1-t)^2}}=\cfrac{1}{1-3t+t^2}$$
which is the generating function of the sequence A001906 (1, 3, 8, 21, 55, 144, 377, 987, 2584,$\ldots$) in OEIS, \cite{Sloane} starting from 1, i.e., the bisection of Fibonacci sequence.\\
\indent In the following example, we examine what happens when the coloration is given by the Catalan numbers A000108.
\begin{example}
Let us consider the sequence $$(C_n)_{n=0}^\infty=(1, 1, 2, 5, 14, 42, 132, 429, 1430,\ldots)$$ of the Catalan numbers. For the coloration $X=(x_1, x_2,\ldots)$, where $x_i=C_{i-1}$ for every $i\geq1$, it is well--known that (see \cite{Cameron}) $$I(X)=(1, 2, 5, 14, 42, 132, 429, 1430,\ldots).$$ We obtain that the number of colored compositions for the integer $n$ is $C_n$, where each part $i$ can occur in $C_{i-1}$ different ways. For example, let us consider the compositions of 4. We have one color for 1, one color for 2, two colors for 3, labeled as $3_1$ and $3_2$, five colors for 4, labeled as $4_1$, $4_2$, $4_3$, $4_4$ and $4_5$. Thus there are $14$ colored compositions of $4$:
$$1111 \quad 112 \quad 121 \quad 211 \quad 22 \quad 13_1 \quad 3_11\quad 13_2\quad 3_21 \quad 4_1 \quad 4_2 \quad 4_3 \quad 4_4 \quad 4_5 .$$
As pointed out in \cite{Stan}, this is one of the  combinatorial interpretations for the Catalan numbers.
\end{example}
In \cite{bcm} some interesting results about the action of the Invert operator on linear recurrence sequences have been proved. We recall them in the next theorems without proof.
\begin{theorem}
Let $a=(a_n)_{n=0}^\infty$ be a linear recurrence sequence of degree $r$, that is, $a_n=h_1a_{n-1}+ \cdots +h_ra_{n-r}$, with characteristic polynomial given by $f(t)=t^r-\sum_{i=1}^{r}h_it^{r-i}$, and generating function $a(t)=\cfrac{u(t)}{f^R(t)}\quad$, where  $f^R(t)=1-\sum_{i=1}^{r}h_{i}t^{i}$  denotes the reflected polynomial of $f(t)$ and $u(t)$ is a polynomial whose coefficients can be obtained from the initial conditions. Then $b=I^{(x)}(a)$ is a linear recurrence sequence, with characteristic polynomial $$(f^R(t)-xtu(t))^R,$$ whose coefficients are
\begin{equation}\label{Ifc}
h_1+xa_0, \quad h_{i+1}  + x\left(a_i  - \sum\limits_{j = 1}^{i} {h_j } a_{i - j}\right) \quad \text{for} \ i=1,\ldots,r-1\quad.
\end{equation}
\end{theorem}
\begin{theorem}
The sequence $b=I^{(x)}(a)$ is characterized by the following recurrence
\begin{equation}\label{invertreceq}
\left\{ \begin{array}{l}
b_n  = a_n  + x\sum\limits_{j = 0}^{n - 1} a_{n - 1 - j} b_j \quad \text{for} \quad n\geq1; \\ 
b_0  = a_0.\\ 
\end{array} \right. 
\end{equation}
\end{theorem}
\noindent The previous theorems are a helpful tool for managing the action of the Invert operator on linear recurrence sequences. They allow us to easily describe the sequence $A(X)$ when the coloration $X$ is described by a linear recurrence sequence. We emphasize some applications of this fact in the following propositions.
\begin{proposition}
The number of colored compositions of a non--negative integer $n$ having $k$ possible colors for any part is $k(k+1)^{n-1}$.
\end{proposition}
\begin{proof}
The sequence of colors is $X=(k,k,k,k,k,\ldots)$, which is a linear recurrence sequence with characteristic polynomial $t-1$ and initial condition $k$. By Eq. (\ref{Ifc}) we immediately have that $A(X)=I(X)$ is a linear recurrence sequence with characteristic polynomial $t-(k+1)$ and by Eq. (\ref{invertreceq}) the initial condition is $k$.
\end{proof}
\begin{proposition}
Forbidding the use of the integer $k$, the number of compositions of a non--negative integer $n$ are $a_{n+1}$, where
\begin{equation} \label{nok}\begin{cases} a_n=2a_{n-1}-a_{n-k}+a_{n-(k+1)},\quad \forall n\geq k+1 \cr a_{i}=2^{i},\quad i=0,\ldots,k-1 \cr a_k=2^k-1  \end{cases}\end{equation}
\end{proposition}
\begin{proof}
Under this condition, the coloration of the compositions is the sequence
$$X=(\underbrace{1,\ldots,1}_{k-1},0,1,1,1,\ldots),$$
which is a linear recurrence sequence with characteristic polynomial $t^{k+1}-t^k$ and initial conditions $(1,\ldots,1,0,1)$. Thus, by Eq. \ref{Ifc}, $I(X)$ is a linear recurrence sequence with characteristic polynomial $t^{k+1}-2t^k+t-1$ and, by Eq. (\ref{invertreceq}), it has initial conditions $(1,2,\ldots,2^{i},\ldots,2^{k-1},2^k-1)$.
\end{proof}
\begin{remark}
The compositions studied in the previous proposition can be found in \cite{Chia} (for the case $k=2$) and \cite{Chi} (for the general case), where the recurrence (\ref{nok}) has been proved in a different way.
\end{remark}

Another consequence of Eq. (\ref{inva}) is summarized in the following
\begin{theorem}\label{invinv}
For every sequence $a=(a_n)_{n=1}^\infty$ with generating function $a(t)$, $a_n$ is the number of colored compositions of $n$, $\forall n\geq1$, with coloration $X$ if and only if $I^{(-1)}(a)$ is a sequence of non--negative integers and in this case $X=I^{(-1)}(a)$, where $I^{(-1)}(a)$ has generating function $\frac{a(t)}{1+ta(t)}.$
\end{theorem}
Thus, the sequence of natural numbers clearly does not represent any colored compositions, since
$$I^{(-1)}(1, 2, 3, 4, 5, 6,\ldots)=(1, 1, 0, -1, -1, 0,\ldots).$$
\begin{remark}
Using Theorem \ref{invinv}, we can easily answer a question posed in \cite{Heu} (Research Direction 3.4, page 88). Translating it to terms of colored compositions, the kernel of the question is finding for which colorations the Fibonacci numbers $(F_n)_{n=0}^\infty=(0, 1 ,1 ,2 ,3 ,5 ,8 ,\ldots)$ count the number of some colored compositions. It is well--known that if $X=(1,1,0,0,0,\ldots)$, i.e., we only have one choice for integer 1 and one choice for integer 2, then the number of compositions for an integer $n$ is $F_{n+1}$ and, indeed, we have
$$I^{(-1)}(1,2,3,5,8,13,\ldots)=(1,1,0,0,0,\ldots).$$
Moreover from the equality
$$I^{(-1)}(1,1,2,3,5,8,13,\ldots)=(1, 0, 1, 0, 1, 0, 1, 0,\ldots).$$
we have only odd integers with one color, then any integer $n$ can be composed in $F_n$ different ways and
$$I^{(-1)}(0,1,1,2,3,5,8,13,\ldots)=(0, 1, 1, 1, 1, 1, 1, 1,\ldots).$$
By Eq. (\ref{invertreceq}), for every sequence $a=(a_n)_{n=0}^\infty$, we have $(I^{(-1)}(a))_1=a_1-a_0^2$. Now if $a_1=F_n$ for some $n$, then $a_1-a_0^2= F_n-F_{n-1}^2<0$ for $n\geq 3$, so no such coloration $X$ can exist such that $A(X)$ is the Fibonacci sequence shifted by three or more steps. Summarizing we uniquely have the following possibilities
$$X=(0,1,1,1,1,1,\ldots),\quad A(X)=(F_n)_{n=0}^\infty$$ 
$$X=(1,0,1,0,1,0,\ldots),\quad A(X)=(F_{n+1})_{n=0}^\infty$$
$$X=(1,1,0,0,0,0,\ldots),\quad A(X)=(F_{n+2})_{n=0}^\infty.$$
\end{remark}
\noindent Furthermore, using the Invert operator we can provide a connection in terms of colored compositions among the $r$--bonacci numbers. We recall that the sequence $F^{(r)}=(F^{(r)}_n)_{n=1}^\infty$ of the $r$--bonacci numbers is the linear recurrence sequence with initial conditions $\underbrace{(0,\ldots,0,1)}_r$ and characteristic polynomial $t^r-t^{r-1}-\cdots-t-1$ (for $r=2$ we have the Fibonacci numbers A000045, for $r=3$ the Tribonacci numbers A000073, etc $\ldots$). The  interesting result proved in \cite{bcm} that
$$I^{(-1)}(F^{(r+1)})=(0,F^{(r)}_1,F^{(r)}_2,F^{(r)}_3,...),$$
interpreted in terms of colored compositions, shows that the $r$--bonacci number $F_n^{(r)}$ corresponds to the number of colored compositions of $n+1$, where a part of length $k$ can be chosen from among $F^{(r+1)}_k$ possibilities (colors).\\
\indent Finally, the relations between the Invert operator and colored compositions provide a new formula for the Invert operator. Indeed, it is well--known that the sequence $A(X)$ is a linear recurrence sequence \cite{count}:
$$A_m(X)=x_1A_{m-1}(X)+x_2A_{m-2}(X)+...+x_m.$$
Thus, from (\ref{inva}) we plainly obtain
$$(I(X))_m=x_1(I(X))_{m-1}+x_2(I(X))_{m-2}+...+x_m,$$
where $X=(x_1,x_2,...)$ can be any integer sequence.

\section{Elegant compositions with the ``black tie"}

\noindent In this section we study colored compositions where we impose a restriction on the coloration. We now take the viewpoint of a composition as a tiling, where each part of size $k$ in the composition corresponds to a tile of size $1\times k$. This device  will allow us to obtain interesting combinatorial results as a new interpretation for the convolution product between sequences. Let us suppose that the black color does not belong to the coloration $X$, i.e., there are no black parts.
\begin{definition}
For all the colorations $X$, we define $M_n(X)$ as the set of colored compositions of $n$ with coloration $X$ such that any composition $\textbf{b}\in M_n(X)$ has exactly one part of length 1 and color black (\emph{black square}). We call \emph{tiling} a black tie composition belonging to $M_n(X)$. Moreover, we indicate $\lvert M_n(X)\rvert=B_n(X)$.
\end{definition}
Many sequences, which could not represent colored compositions, obtain a new interpretation associated to black tie compositions. For example, as we pointed out in the previous section, the sequence of natural numbers does not represent any colored composition, but clearly represents the black tie composition with coloration $X=(1,0,0,0,0,\ldots)$, i.e., $(B_n(X))_{n=1}^\infty=(1,2,3,4,\ldots)$.
\begin{theorem}\label{t4}
Given the coloration $X$, we have 
$$\begin{cases} B_0(X)=0,\quad B_1(X)=1,\cr B_n(X)=R_{n-1}(X)\quad \text{for}\quad n\geq2\quad . \end{cases}$$
\end{theorem}
\begin{proof}
It is straightforward to see that $B_0(X)=0$ and $B_1(X)=1$. Given $\textbf{b}\in L_{n-1}(X)$, if we insert a black square into any break point of $\textbf{b}$, then we obtain a tiling in $M_n(X)$ and obviously every tiling in $M_n(X)$ is formed in this way, thus $B_n(X)=R_{n-1}(X)$, for $n\geq2$.
\end{proof}
Since $R_n(X)=P_n(X)+A_n(X)$ for $n\geq1$ for any coloration $X$, we obtain the following result.
\begin{corollary}
For $n\geq 1$,  $B_n(X)=P_{n-1}(X)+A_{n-1}(X).$
\end{corollary}

Now, we point out new and interesting combinatorial results, exploiting these relations between colored compositions and black tie compositions. First of all, we recall that the number of colored compositions of the integer $n$ with coloration $X=(x_1,x_2,\ldots)$ has an explicit formula involving the multinomial coefficients:
\begin{equation} \label{Amult} A_n(X)=\sum_{k_1+2k_2+\cdots +nk_n=n}(k_1,...,k_n)!x_1^{k_1}\cdots x_n^{k_n},\end{equation}
where $(k_1,...,k_n)!$ is the multinomial coefficient
$$(k_1,...,k_n)!=\cfrac{(k_1+...+k_n)!}{k_1!\cdots k_n!}.$$
This equation simply underlines that for a composition of the integer $n$ we are using $k_1$ parts of size 1 chosen from $x_1$ different colors, $k_2$ parts of size 2 chosen from $x_2$ different colors, etc$\ldots$ (see, e.g., \cite{count}).
\begin{remark}
Eqs. (\ref{inva}) and (\ref{Amult}) provide a new closed formula for the Invert operator.
\end{remark}
\begin{theorem}\label{t5}
Given $X=(x_1,x_2,\ldots)$, we have that
$$B_{n+1}(X)=\sum_{k_1+2k_2+\cdots+nk_n=n}(k_1,...,k_n)!x_1^{k_1} \cdots x_n^{k_n}(1+\sum_{j=1}^nk_j).$$
\end{theorem}
\begin{proof}
Every composition $\textbf{b}\in L_n(X)$ generates $r(\textbf{b})$ tilings in $M_{n+1}(X)$ (inserting a black square in the positions of the break points of $\textbf{b}$). Thus, from Eq. (\ref{Amult}), the claim follows since $r(\textbf{b})=p(\textbf{b})+1$ and the number of parts of each composition is $\sum_{j=1}^nk_j$.
\end{proof}

\begin{theorem}\label{t6}
Given $X=(x_1,x_2,\ldots)$, we have
$$B_{n}(X)=\sum_{k_1+2k_2+\cdots +(n-1)k_{n-1}=n}(k_1,\ldots,k_n)!x_1^{k_1-1}x_2^{k_2} \cdots x_{n-1}^{k_{n-1}}k_1.$$
\end{theorem}
\begin{proof}
Let $\textbf{b}$ be a composition in $L_n(X)$ that has $k_1$ parts of length 1, $k_2$ parts of length 2,$\ldots$, $k_n$ parts of length $n$. We know that the possible number of compositions $\textbf{b}$ of this kind is
$$(k_1,\ldots, k_n)!x_1^{k_1} \cdots x_n^{k_n}.$$
To create a tiling in $M_n(X)$ from a composition $\textbf{b}\in L_n(X)$, we have to replace a part of size 1 in $\textbf{b}$ with a black square. The number of compositions with $k_1-1$ parts of length 1, $k_2$ parts of length 2,$\ldots$, $k_n$ parts of length $n$ is
$$(k_1,\ldots, k_n)!x_1^{k_1-1} \cdots x_n^{k_n}.$$
Moreover, this number must be multiplied by $k_1$ that is the number of possible replacements of a part of length 1 with the black square:
$$(k_1,\ldots, k_n)!x_1^{k_1-1}\cdots x_n^{k_n}k_1.$$
So the claim easily follows, observing that in the sum 
$$\sum_{k_1+2k_2+\cdots +(n-1)k_{n-1}+nk_n=n}(k_1,\ldots, k_n)!x_1^{k_1-1}x_2^{k_2} \cdots x_{n-1}^{k_{n-1}}x_n^{k_n}k_1$$
the only term containing $x_n$ corresponds to the choice $k_1=\cdots=k_{n-1}=0$ and $k_n=1$ and consequently $k_n$ and $x_n$ can be omitted from the sum.

\end{proof}

\begin{corollary}
Given $X=(x_1,x_2,\ldots)$, if we consider $A_n(X)$ as a polynomial in the variables $x_1,x_2,\ldots $, then
$$B_n(X)=\cfrac{\partial A_n(X)}{\partial x_1},\quad \forall n\geq1.$$
\end{corollary}

Combining Theorems \ref{t5} and \ref{t6}, we have proved, only using combinatorial arguments, the following remarkable identity involving multinomial coefficients.

\begin{theorem}
For non--negative integers $x_1,\ldots,x_n$ and non--negative integers $k_1, \ldots, k_n$, we have
\small
$$\sum_{k_1+\cdots+nk_n=n}(k_1,\ldots, k_n)!x_1^{k_1} \cdots x_n^{k_n}\left(1+\sum_{j=1}^nk_j\right)=$$
$$=\sum_{k_1+\cdots +nk_n=n+1}(k_1,\ldots, k_n)!x_1^{k_1-1}x_2^{k_2} \cdots x_n^{k_{n}}k_1$$
\end{theorem}
\normalsize
\begin{remark}
Black tie compositions induce a new transform on integer sequences: a sequence $X$ is changed into a sequence $B(X)$ by means of the closed formulas proved in Theorems \ref{t5} and \ref{t6}.
\end{remark}
The study of black tie compositions provides the first combinatorial interpretation for the \emph{conv} transformation of sequences defined in  \cite{Ber}. As a consequence we can develop a new and very useful method to evaluate the number of parts of colored compositions. Before stating the Theorem concerning these results, we underline that in the following with $A$ or $A(X)$ we refer to the sequence $(A_n)_{n=0}^\infty$ starting from 0. Moreover, we recall the action of the right-shift operator on a sequence $a$.
\begin{definition} \label{sigma}
The \emph{right-shift operator} $\sigma$ changes any sequence $a=(a_0,a_1,a_2,\ldots)$ as follows:
\begin{equation*} 
\sigma(a)=(a_1,a_2,a_3,\ldots)
\end{equation*}
\end{definition}
\begin{theorem}\label{t8}
Given a coloration $X$, the sequence $B(X)$ is the right shift $\sigma$ of the convolution product $*$ of $A(X)$ with itself
$$B(X)=\sigma(A(X)*A(X)).$$
\end{theorem}
\begin{proof}
Remembering that $L_0(X)$ is the empty set, we construct  every tiling of $M_{n+1}(X)$ by joining together a composition $\textbf{b}\in L_k(X)$, the black square, and a second composition $\textbf{c}\in L_{n-k}(X)$, where $k$ ranges between 0 and $n$. Thus, the number of tilings in $M_{n+1}(X)$ is $$B_{n+1}(X)=\sum_{k=0}^n A_{n-k}(X)A_k(X).$$
\end{proof}
This theorem provides a combinatorial interpretation in terms of counting black tie compositions for the operator \emph{conv}. In fact we recall that this operator maps a sequence $(a_n)_{n=0}^\infty$ into a sequence $(b_n)_{n=0}^\infty$ by means of
$$b_n=\sum_{k=0}^na_{n-k}a_k,\quad n=0,1,...$$
\begin{corollary} The generating function for $B(X)$ is $\frac{tf(t)^2}{(1-tf(t))^2}$ where $f(t)$ is the generating  function of $X$.
\end{corollary}
\begin{proof} This clearly follows observing that by Theorem \ref{t4} and Theorem \ref{t8}
$$\sum_{n\geq0}B_n(X)t^n=\sum_{n\geq0}B_{n+1
}(X)t^{n+1}=t\sum_{n\geq0}\left( \sum_{k=0}^n A_{n-k}(X)A_k(X)\right)t^n=t[A(X)(t)]^2$$ where $ A(X)(t)$ is the generating function of $A(X)=I(X)$. Now, thanks to Definition 2, the claim is straightforward.
\end{proof}
\begin{corollary} \label{parts} Given a coloration $X$, we have
\begin{equation}\label{partscount} P(X)=A(X)*A(X)-A(X). \end{equation}
\end{corollary}
\begin{proof}
Since $B_n(X)=P_{n-1}(X)+A_{n-1}(X)$, the claim is an immediate consequence of the previous theorem.
\end{proof}
Eq. \ref{partscount} is a powerful tool in order to find the number of parts of black--tie colored compositions. First of all, for all  colorations $X$, we are always able to count, in an easy and fast way, the number of such compositions of any integer by means of the Invert operator through Eq. (\ref{inva}). Then, we can count the number of parts of such compositions of any integer by means of the convolution product of sequences.
\begin{example}
Let us consider the coloration $X=(1,1,0,0,0,0,\ldots)$, which results in $A_n(X)=F_{n+1}$. Since $A_0(X)=1$, we have
$$A*A=(1,2,5,10,20,38,\ldots)$$
which is the sequence A001629 starting from 1. Thus, by Corollary \ref{parts} we have that the number of parts of these colored compositions is given by the sequence 
$$A*A-A=(0, 1, 3, 7, 15, 30, 58, 109, 201, 365,\ldots)$$
which is $\sigma(A023610)$. For example the number of parts of colored compositions of 5 is 30:
$$11111\quad 1112\quad 1121 \quad 1211 \quad 2111 \quad 221 \quad 212 \quad 122$$
\end{example}


\begin{example}
Let us consider the coloration with Catalan numbers as in Example 1, $X=(1,1,2,5,14,42,\ldots)$. Using Corollary \ref{parts}, the number of parts for these colored compositions is 
$$P(X)=(0,1, 3, 9, 28, 90, 297, 1001, 3432, 11934,\ldots)$$
which corresponds to the sequence A000245. 
\end{example}

\begin{example}
Starting from the coloration $X=(1,1,1,1,1,1,\ldots)$, we obtain $A(X)=(1, 1, 2, 4, 8, 16, 32, 64, 128, 256,\ldots)$, the number of compositions of $n$. We can immediately evaluate
$$A*A=(1, 2, 5, 12, 28, 64, 144, 320, 704,\ldots)$$
and the sequence of the number of parts is
$$P=A*A-A=(0, 1, 3, 8, 20, 48, 112, 256, 576, 1280,\ldots).$$
In \cite{Sloane}, the sequence $(1, 3, 8, 20, 48, 112, 256, 576, 1280 \ldots)$ is  A001792.
\end{example}

\end{document}